\documentclass[A4,11pt]{article}
\usepackage[utf8]{inputenc}
\usepackage{geometry}
\usepackage{graphicx}
\usepackage{amssymb}
\usepackage{amsthm}
\usepackage{amsmath}
\usepackage{amsfonts}
\usepackage{ebproof}
\usepackage{xcolor}
\usepackage{enumitem}
\usepackage{varwidth}
\usepackage{setspace} 
\usepackage{xspace}

\usepackage[backend=biber,style=authoryear,sorting=nyt]{biblatex}
\theoremstyle{definition}
\newtheorem{thm}{Theorem}
\newtheorem{lem}[thm]{Lemma}

\newtheorem{defn}[thm]{Definition}

\usepackage{hyperref}
\hypersetup{
    colorlinks=true,
    linkcolor=blue,
    citecolor=blue,
    urlcolor=blue,
}

\newcommand{\set}[1]{\{#1\}}

\newcommand{\powerset}{\mathcal{P}}
\newcommand{\axiom}[1]{\bar{#1}}
\newcommand{\inference}[2]{\frac{#1}{#2}}
\newcommand{\poscite}[2][]{\citeauthor{#2}'s (\citeyear[#1]{#2})}

\numberwithin{equation}{section}
\numberwithin{enumi}{section}
\numberwithin{thm}{section}
\numberwithin{figure}{section}

\newcommand{\boldchange}[1]
{#1}
\usepackage{bm}

\usepackage{tikz}
\usetikzlibrary{arrows,calc,patterns,positioning,shapes,shapes.misc}
\usetikzlibrary{decorations.pathmorphing}
\tikzset{
modal/.style={shorten >=1pt,shorten <=1pt,auto,
node distance=.5cm,semithick},
world/.style={circle,draw,minimum size=.7cm,fill=gray!15},
point/.style={circle,draw,fill=black,inner sep=0.5mm},
reflexive/.style={->,in=120,out=60,loop,looseness=#1},
reflexive/.default={5},
reflexive point/.style={->,in=135,out=45,loop,looseness=#1},
reflexive point/.default={25},
}
\tikzset{cross/.style={cross out, draw, 
         minimum size=2*(#1-\pgflinewidth), 
         inner sep=0pt, outer sep=0pt}}

\usepackage{xcolor}

\usepackage{todonotes}

\title{Following all the rules: Intuitionistic completeness for generalised proof-theoretic validity}
\author{Will Stafford and Victor Nascimento
\thanks{Dr Stafford would like to thank Sean Walsh the audiences at Orange County and Inland Empire History and Philosophy of Mathematics and Logic and MCMP Colloquium in Mathematical Philosophy for their helpful comments. He was supported by a Lumina quaeruntur fellowship, LQ300092101, from the Czech Academy of Sciences. Mr Nascimento would like to thank Luiz Carlos Pereira. His work was financed in part by the Coordenaç\~ao de Aperfeiçoamento de Pessoal de N\'ivel Superior – Brasil (CAPES) – Finance Code 001. Both authors would like to thank the two anonymous reviewers.}}
\date{\today}

\begin{document}
\maketitle

\begin{abstract}
    Prawitz conjectured that the proof-theoretically valid logic is intuitionistic logic.  Recent work on proof-theoretic validity has disproven this.  In fact, it has been shown that proof-theoretic validity is not even closed under substitution. In this paper, we make a minor modification to the definition of proof-theoretic validity found in \textcite{Prawitz1973-hc} and refined by \textcite{Schroeder-Heister2006-rh}. We will call the new notion generalised proof-theoretic validity and show that the logic of generalised proof-theoretic validity is intuitionistic logic. 
\end{abstract}

\section{Introduction}
Prawitz developed proof-theoretic validity in the early 1970s as a method of demonstrating that intuitionistic elimination rules follow from intuitionistic introduction rules. Prawitz conjectured not only that intuitionistic elimination rules follow according to proof-theoretic validity but also that no stronger elimination rules follow from then. This amounts to conjecturing that intuitionistic logic is sound and complete when proof-theoretic validity is treated as a semantics. 

This conjecture was firmly refuted by \textcite{Piecha2019}, who demonstrated that most varieties of proof-theoretic validity are actually stronger than intuitionistic logic. This includes Prawitz's 1970s notion and his later \citeyear{Prawitz2006-dr} proposal. The goal of this paper is to show that \boldchange{propositional} intuitionistic logic is sound and complete for generalised proof-theoretic validity, which results from a small modification of the 1970s definition:
\begin{thm}
\boldchange{For all $\varphi$ in the language of propositional logic,} $\varphi$ is a generalised proof-theoretically valid formula $\Leftrightarrow$ $\varphi$ is an intuitionistic validity.
\end{thm}
\noindent This result generalises \poscite{Goldfarb2016-tk} revision of proof-theoretic validity, leading to a notion which, unlike Goldfarb’s, is closed under uniform substitution. 

The insight here is as follows: \textcite{Piecha2019} propose that a set of inference rules for atomic propositions is the proof-theoretic equivalent of a model, which is why validity should be defined over a collection of such sets. We propose that a set of inference rules is equivalent to a world in a model, which is why validity should be defined over a collection of sets of sets. If this is done, the resulting logic is intuitionistic. 

This result in a sense vitiates Prawitz's conjecture. The early treatment of atomic propositions failed to yield a system closed under substitution, which suggests that there is some technical issue with the implementation of the proposal. We take generalised proof-theoretic validity to be the natural approach to ensuring closure under substitution. And this leads to a notion for which intuitionistic logic is sound and complete.

The first section of this paper introduces proof-theoretic validity, the second recalls intuitionistic Kripke models, and the proof is given in the third and final section.
\section{Generalised proof-theoretic validity}
In this section, we lay down the definition of proof-theoretic validity that will be used. We will use the definition of the proof-theoretic consequence relation found in \textcite{Piecha2015} rather than working directly with the definition on proof-like structures. Let us start with the treatment of atomic propositions. Note that we treat $\bot$ as an atomic proposition, not as a 0-ary connective.

\begin{defn}
An atomic rule is either an axiom of the form $\begin{prooftree}\hypo{}\infer1{p}\end{prooftree}$ for any $p\in ATOM\cup\set{\bot}$ or an inference of the form $\begin{prooftree}\hypo{p_1\dots p_n}\infer1{p_{n+1}}\end{prooftree}$ for any $p_1,\dots,p_{n+1}\in ATOM\cup\set{\bot}$. The set of all atomic rules will be denoted as $\mathbb{S}$.
\end{defn}
Let an \emph{atomic system} $S$ be a subset of $\mathbb{S}$. A \emph{proof-theoretic system} is then any $\mathfrak{S}\subseteq\powerset(\mathbb{S})$. The consequence relation between atomic systems and atomic propositions can then be defined as follows:
\begin{defn}
Given an atomic system $S$ and an atomic proposition $p$, we will write $S\vdash p$ if there is a proof of $p$ using only rules in $S$.
\end{defn}
With all this in place, we can define proof-theoretic validity:

\begin{defn}\label{def:piecha_logic}
The \emph{proof-theoretic validity consequence relation} $\vDash$ is such that for every $\mathfrak{S}$ and $S\in\mathfrak{S}$:
\begin{align}
    \mathfrak{S},S\vDash p&\Longleftrightarrow S\vdash p,\\
    \mathfrak{S},S\vDash \bot&\Longleftrightarrow S\vdash\bot,\\
    \mathfrak{S},S\vDash\varphi\wedge\psi&\Longleftrightarrow \mathfrak{S},S\vDash\varphi\text{ and }\mathfrak{S},S\vDash\psi,\\
    \mathfrak{S},S\vDash\varphi\vee\psi&\Longleftrightarrow \mathfrak{S},S\vDash\varphi\text{ or }\mathfrak{S},S\vDash\psi, \\
    \mathfrak{S},S\vDash\psi\rightarrow \varphi&\Longleftrightarrow[\forall S'\in\mathfrak{S}(S'\supseteq S\text{ and }\mathfrak{S},S'\vDash\psi\Rightarrow\mathfrak{S},S'\vDash\varphi)]. 
\end{align}
\noindent Further, let $\mathfrak{S}\vDash\varphi$ hold if and only if $\mathfrak{S},S\vDash\varphi$ for all $S\in\mathfrak{S}$.
\end{defn}

\noindent Most presentations of proof-theoretic validity suppress any reference to the proof-theoretic system. The importance of explicitly stating the proof-theoretic system $\mathfrak{S}$ is recognised by \textcite{Piecha2015}, who highlight the differences between restrictions put on permitted sets of atomic rules in the literature and show how these differences affect what is valid. \textcite{Piecha2019} posit that proof-theoretic systems are analogous to the collection of models relative to which  model-theoretic consequence relations are defined. Given this understanding, their result that every proof-theoretic system is super-intuitionistic can be interpreted as showing that there are no treatments of the atomic propositions that are intuitionistic, and therefore proof-theoretic validity is not intuitionistic either\footnote{\boldchange{\textcite{Piecha2019} are careful to point out that there are other ways to define the atomic formulas which avoid their result such as \textcite{Goldfarb2016-tk}. We discuss this connection at the end of Section~\ref{sec:kripke2pts}.}}. As mentioned above, we are guided by the idea that a proof-theoretic system is analogous to a model, not to a collection of models. This allows us to view \poscite{Piecha2019} result as demonstrating instead that no individual ``model'' is intuitionistic, which is no odder than pointing out that every classical model either models $p$ or $\neg p$, but neither is a classical validity.

The largest proof-theoretic system is $\powerset(\mathbb{S})$. This proof-theoretic system is ``minimal" in the sense that $\bot$ is not defined and intuitionistic logic is therefore not sound. \boldchange{If an atomic system contains a rule $\inference{\bot}{p}$ for every atomic proposition $p$, then $\bot$ will behave as though defined by its elimination rule.}  We will call a proof-theoretic system $\mathfrak{S}\subseteq\powerset(\mathbb{S})$ intuitionistic if every $S\in\mathfrak{S}$ contains $\inference{\bot}{p}$ for every $p$. It is known that intuitionistic logic is sound on the resulting systems.

We can now define generalised proof-theoretic validity as follows:
\begin{defn}
$\varphi$ is a \emph{generalised proof-theoretically valid formula} if for every intuitionistic proof-theoretic system $\mathfrak{S}$, it follows that $\mathfrak{S}\vDash\varphi$. (That is, $GPTV=\set{\varphi\mid\forall\mathfrak{S}\subseteq\powerset(\mathbb{S})[\mathfrak{S}$ intuitionistic $\rightarrow\mathfrak{S}\vDash\varphi]}$.)
\end{defn}
This definition differs from those considered by \textcite{Piecha2019} because we have not chosen a particular proof-theoretic system to define proof-theoretic validity over. Our goal is to show that $INT=GPTV$, i.e., the set of intuitionistic validates, coincides with generalised proof-theoretically valid formulas. 

\section{Kripke models}
In this section, we lay out the definition of an intuitionistic Kripke model. It is already known that every intuitionistic proof-theoretic system is equivalent to a Kripke model (\cite{Piecha2016-gt}). 

\boldchange{Recall that, in this context, a partial order is a relation $R$ that is transitive, antisymmetric, and reflexive. Moreover, a function $f$ on a partial order is monotonic with respect to the subset relation if $R(a,b)$ implies $f(a)\subseteq f(b)$. This can be understood as a condition preventing one from ``changing ones mind" when transitioning from a world $a$ to an accessible world $b$, since everything assigned by the function to $a$ will also be assigned to $b$.}
\begin{defn}
An \emph{intuitionistic Kripke model} $\mathcal{M}=\langle\langle W,R\rangle, V\rangle$ is a Kripke frame $\langle W,R\rangle$ consisting of a set $W$ of worlds and an accessibility relation $R\subseteq W\times W$ that is a partial order, plus a monotonic valuation function $V:W\rightarrow ATOM$.
\end{defn}

\begin{defn}\label{def:kripke_logic}
Define $\Vdash$ on pairs consisting of an intuitionistic Kripke model $\mathcal{M}=\langle\langle W,R\rangle, V\rangle$ and a world $w\in W$:
\begin{align}
    \mathcal{M},w\Vdash p&\Longleftrightarrow p\in V(w),\\
    \mathcal{M},w\nVdash \bot&,\\
    \mathcal{M},w\Vdash\varphi\wedge\psi&\Longleftrightarrow \mathcal{M},w\Vdash\varphi\text{ and }\mathcal{M},w\Vdash\psi,\\
    \mathcal{M},w\Vdash\varphi\vee\psi&\Longleftrightarrow \mathcal{M},w\Vdash\varphi\text{ or }\mathcal{M},w\Vdash\psi, \\
    \mathcal{M},w\Vdash\psi\rightarrow \varphi&\Longleftrightarrow[\forall w'\in W(Rww'\text{ and }\mathcal{M},w'\Vdash\psi\Rightarrow\mathcal{M},w'\Vdash\varphi)]. 
\end{align}
\end{defn}


\section{From finite Kripke models to proof-theoretic systems}\label{sec:kripke2pts}
We will demonstrate that every finite Kripke model is equivalent to an intuitionistic proof-theoretic system. 
\begin{defn}
An intuitionistic Kripke model is finite if $W$ is finite.
\end{defn}
\noindent We must restrict the size of Kripke models because while a Kripke model can be arbitrarily large, a proof-theoretic system is bounded by $|\powerset(\mathbb{S})|$ which – given that we will, in general, have a countable infinity of atomic propositions – will be the cardinality of the reals. This problem cannot be solved by adding more atomic propositions because we would need a proper class of atomic propositions to ensure that there is a model of every cardinality, which is something we take to be unreasonable.

Showing that every finite Kripke model is equivalent to an intuitionistic proof-theoretic system will suffice for demonstrating that intuitionistic logic is complete for generalised proof-theoretic validity because of the following result:
\begin{thm}[{\cite[Theorem 6.12]{Troelstra1988-mb}}]\label{thm:finitecomplete}
Intuitionistic logic is complete for the class of all finite intuitionistic Kripke models.
\end{thm}

\boldchange{Now that we are considering only finite intuitionistic Kripke models, it might seem natural to simply try and reverse the obvious method of generating intuitionistic Kripke models from Kripke-like proof-theoretic systems by letting $S_w=\set{\axiom{p}\mid p\in V(w)}$. In fact, this actually works in some cases:}

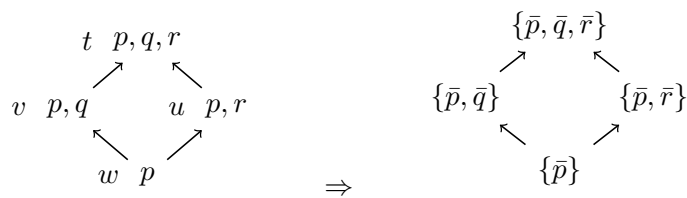
\begin{center}
    \begin{tikzpicture}[modal]
        \node (c) [] {};
        \node (w) [label=left:{$w$},below=of c] {$p$};
        \node (v) [label=left:{$v$},left=of c] {$p,q$};
        \node (u) [label=left:{$u$},right=of c] {$p,r$};
        \node (t) [label=left:{$t$},above=of c] {$p,q,r$};
        \path[->] (w) edge (v);
        \path[->] (w) edge (u);
        \path[->] (u) edge (t);
        \path[->] (v) edge (t);
    \end{tikzpicture}    $\qquad\Rightarrow\qquad$
    \begin{tikzpicture}[modal]
        \node (c) [] {};
        \node (w) [below=of c] {$\set{\axiom{p}}$};
        \node (v) [left=of c] {$\set{\axiom{p},\axiom{q}}$};
        \node (u) [right=of c] {$\set{\axiom{p},\axiom{r}}$};
        \node (t) [above=of c] {$\set{\axiom{p},\axiom{q},\axiom{r}}$};
        \path[->] (w) edge (v);
        \path[->] (w) edge (u);
        \path[->] (u) edge (t);
        \path[->] (v) edge (t);
    \end{tikzpicture}
\end{center}

\boldchange{However, in many cases this will collapse distinct worlds into the same atomic rules set:}

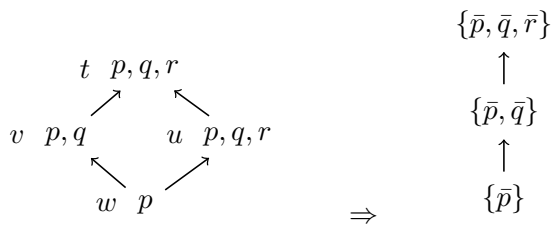
\begin{center}
    \begin{tikzpicture}[modal]
        \node (c) [] {};
        \node (w) [label=left:{$w$},below=of c] {$p$};
        \node (v) [label=left:{$v$},left=of c] {$p,q$};
        \node (u) [label=left:{$u$},right=of c] {$p,q,r$};
        \node (t) [label=left:{$t$},above=of c] {$p,q,r$};
        \path[->] (w) edge (v);
        \path[->] (w) edge (u);
        \path[->] (u) edge (t);
        \path[->] (v) edge (t);
    \end{tikzpicture}    $\qquad\Rightarrow\qquad$
    \begin{tikzpicture}[modal]
        \node (c) [] {$\set{\axiom{p},\axiom{q}}$};
        \node (w) [below=of c] {$\set{\axiom{p}}$};
        \node (t) [above=of c] {$\set{\axiom{p},\axiom{q},\axiom{r}}$};
        \path[->] (w) edge (c);
        \path[->] (c) edge (t);
    \end{tikzpicture}
\end{center}

\boldchange{There is a trick we can pull to resolve this issue.  It involves noting that atomic rules that are not axioms play two distinct roles.  The first is to allow derivations from axioms. The second is to provide structure to the atomic rule sets.  A proof-theoretic system might have two sets that prove the same atomic formulas, say $\set{\axiom{p}}$ and $\set{\axiom{p},\inference{r}{s}}$, but are distinct (and in fact stand in the particular relations they do with regards to the subset relation) because of the atomic inference rules. This gives us a quick but unsystematic fix to the problem above:}
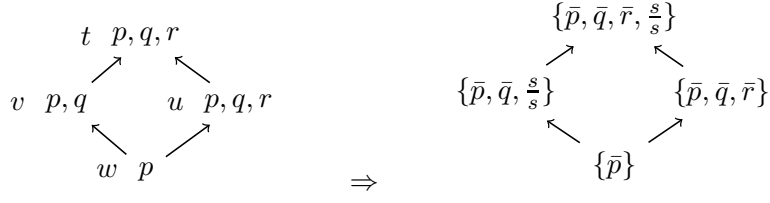
\begin{center}
    \begin{tikzpicture}[modal]
        \node (c) [] {};
        \node (w) [label=left:{$w$},below=of c] {$p$};
        \node (v) [label=left:{$v$},left=of c] {$p,q$};
        \node (u) [label=left:{$u$},right=of c] {$p,q,r$};
        \node (t) [label=left:{$t$},above=of c] {$p,q,r$};
        \path[->] (w) edge (v);
        \path[->] (w) edge (u);
        \path[->] (u) edge (t);
        \path[->] (v) edge (t);
    \end{tikzpicture}    $\qquad\Rightarrow\qquad$
    \begin{tikzpicture}[modal]
        \node (c) [] {};
        \node (w) [below=of c] {$\set{\axiom{p}}$};
        \node (v) [left=of c] {$\set{\axiom{p},\axiom{q},\inference{s}{s}}$};
        \node (u) [right=of c] {$\set{\axiom{p},\axiom{q},\axiom{r}}$};
        \node (t) [above=of c] {$\set{\axiom{p},\axiom{q},\axiom{r},\inference{s}{s}}$};
        \path[->] (w) edge (v);
        \path[->] (w) edge (u);
        \path[->] (u) edge (t);
        \path[->] (v) edge (t);
    \end{tikzpicture}
\end{center}

\boldchange{We can make this systematic by using atomic rules to label worlds.  This may be done by taking a labeling of atomic formulas by worlds $p_w$ for each $w\in W$ and then using $\inference{p_w}{p_w}$ to ensure the atomic rules set does not collapse into any other world.  Because $\inference{p_w}{p_w}$ is tautologous, it will not allow anything new to be derived, and so we do not need to worry about it fulfilling the first role of atomic rules.}

Let $p_{(\cdot)}:W\rightarrow ATOM$ be an injective function from a set of worlds $W$ to the set of atomic propositions. Because $W$ will be finite, we can assume such a function exists. For ease, let us write $p_{(w)}$ as $p_w$. We can now define an intuitionistic proof-theoretic system for every finite intuitionistic Kripke model:

\begin{defn}
Given a finite intuitionistic Kripke model $\mathcal{M}=\langle\langle W,R\rangle, V\rangle$, we define $\mathfrak{S}_{\mathcal{M}}=\set{S_w\mid w\in W}$ as follows: $S_w=\set{\Bar{p}\mid p\in V(w)}\cup\set{\inference{p_s}{p_s}\mid Rsw}\cup\set{\inference{\bot}{p}\mid p \text{ atomic}}$. \boldchange{As with all proof-theoretic systems, $\subseteq$ is the analogue of accessibility relations.}
\end{defn}

\boldchange{Consider a particular atomic system on this interpretation. It will be made up of three parts. The first, $\set{\Bar{p}\mid p\in V(w)}$, ensures that it proves every atomic formula that the world forces. The second, $\set{\inference{\bot}{p}\mid p \text{ atomic}}$, ensures that the system is not minimal. The third, $\set{\inference{p_s}{p_s}\mid Rsw}$, encodes the accessibility relation of the Kripke model.

The following illustrates this method:}
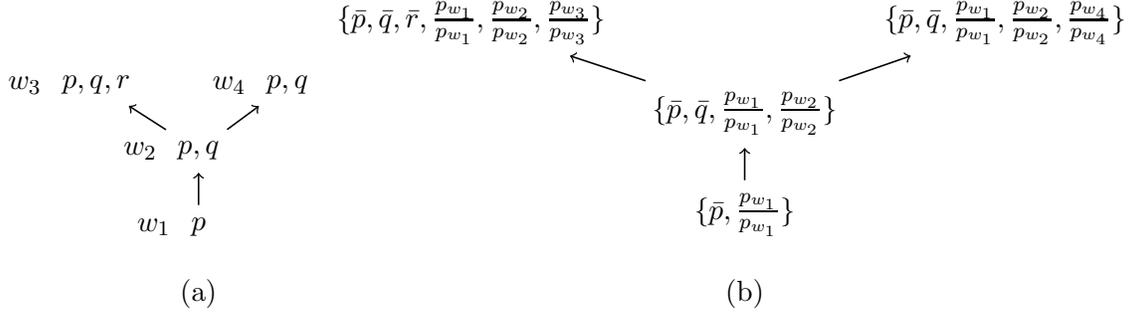
\begin{center}
    \begin{tikzpicture}[modal]
        \node (f) [label=left:{$w_1$}] {$p$};
        \node (e) [label=left:{$w_2$},above=of f] {$p,q$};
        \node (w) [label=left:{$w_3$},above left=of e] {$p,q,r$};
        \node (v) [label=left:{$w_4$},above right=of e] {$p,q$};
        \begin{scope}[node distance=.3cm]
        \node (text) [below=of f] {(a)};
        \end{scope}
        \path[->] (f) edge (e);
        \path[->] (e) edge (v);
        \path[->] (e) edge (w);
    \end{tikzpicture}
    \begin{tikzpicture}[modal]
        \node (f) [] {$\set{\axiom{p},\inference{p_{w_1}}{p_{w_1}}}$};
        \node (e) [above=of f] {$\set{\axiom{p},\axiom{q},\inference{p_{w_1}}{p_{w_1}},\inference{p_{w_2}}{p_{w_2}}}$};
        \node (w) [above left=of e] {$\set{\axiom{p},\axiom{q},\axiom{r},\inference{p_{w_1}}{p_{w_1}},\inference{p_{w_2}}{p_{w_2}},\inference{p_{w_3}}{p_{w_3}}}$};
        \node (v) [above right=of e] {$\set{\axiom{p},\axiom{q},\inference{p_{w_1}}{p_{w_1}},\inference{p_{w_2}}{p_{w_2}},\inference{p_{w_4}}{p_{w_4}}}$};
        \begin{scope}[node distance=.3cm]
        \node (text) [below=of f] {(b)};
        \end{scope}
        \path[->] (f) edge (e);
        \path[->] (e) edge (v);
        \path[->] (e) edge (w);
    \end{tikzpicture}
\end{center}
\boldchange{First of all, note that the world associated with each atomic system is uniquely identified by the propositional letter $p_w$ which encodes it.  However, this alone would not ensure that the proof-theoretic system matched the Kripke model. In order to do that, each atomic system must also encode every world that accesses it: if $Rw_1w_2$, then $S_{w_2}$ contains the propositional letter $p_{w_1}$ encoding $w_1$.  This ensures that the subset relation $\subseteq$ on atomic systems matches the accessibility relation $R$ on worlds.}

We will now demonstrate that $\mathfrak{S}_{\mathcal{M}}$ exists, that it is an intuitionistic proof-theoretic system, and that it models the same formulas as $\mathcal{M}$.
\begin{lem}
For every finite intuitionistic Kripke model $\mathcal{M}$, the set $\mathfrak{S}_\mathcal{M}$ exists and is an intuitionistic proof-theoretic system.
\end{lem}
\begin{proof}
Because $\mathcal{M}$ is finite, we know that there exists a labelling $p_w$ for $w\in W$. It follows that $S_w\subseteq\set{\axiom{p},\inference{p}{p},\inference{\bot}{p}\mid p \text{ atomic}}\subseteq \mathbb{S}$, and therefore $\mathfrak{S}_\mathcal{M}$ is a subset of $\powerset(\mathbb{S})$. Because $\set{\inference{\bot}{p}\mid p \text{ atomic}}\subseteq S_w$, it follows that the system is intuitionistic.
\end{proof}
\noindent The following lemmas demonstrate that $\mathfrak{S}_{\mathcal{M}}$ models the same formulas as $\mathcal{M}$.
\begin{lem}
For every finite intuitionistic Kripke model $\mathcal{M}$, if $w\neq w'$ then $S_w\neq S_{w'}$.
\end{lem}
\begin{proof}
Assume $w\neq w'$ and $S_w=S_{w'}$. It follows that $\set{\inference{p_{w^*}}{p_{w^*}}\mid Rw^*w}=\set{\inference{p_{w^*}}{p_{w^*}}\mid Rw^*w'}$ and because $Rww$ and $Rw'w'$, it follows that $Rww'$ and $Rw'w$. By antisymmetry thus $w=w'$, which is a contradiction.
\end{proof}

\begin{lem}\label{lem:findS}
For every finite $\mathcal{M}$ and $w,w'\in W$ such that $Rww'$, it follows that $S_w\subseteq S_{w'}$.
\end{lem}

\begin{proof}
Let $\tau\in S_w$ be a rule, then either $\tau=\axiom{p}$ for some $p\in V(w)$, $\tau=\inference{\bot}{p}$ for some atomic $p$, or $\tau=\inference{p_{w^*}}{p_{w^*}}$ for some $w^*\in W$ such that $Rw^*w$. In the first case, since $V$ is monotonic and $Rww'$, it follows that $p\in V(w')$ and therefore $\tau=\axiom{p}\in S_{w'}$. In the second case, $\tau=\inference{\bot}{p}\in S_{w'}$ because the system is intuitionistic. In the third case, since $R$ is transitive, $Rw^*w'$, and therefore $\tau=\inference{p_{w^*}}{p_{w^*}}\in S_{w'}$. So $\tau\in S_{w'}$.
\end{proof}

\begin{lem}\label{lem:findw}
For every finite $\mathcal{M}$ and $S_w,S\in\mathfrak{S}_\mathcal{M}$ such that $S_w\subseteq S$, there is a $w'\in W$ such that $S=S_{w'}$ and $Rww'$.
\end{lem}

\begin{proof}
If $S=S_w$, we are done so let us assume not. By the definition of $\mathfrak{S}_\mathcal{M}$, we know that $S=S_{w'}$ for some $w'\in W$ and because $\inference{p_w}{p_w}\in S_w\subseteq S_{w'}$, it follows that $\inference{p_w}{p_w}\in\set{\inference{p_{w^*}}{p_{w^*}}\mid Rw^*w'}$ and therefore $Rww'$.
\end{proof}

\begin{thm}\label{thm:kripketopts}
Given finite $\mathcal{M}=\langle\langle W,R\rangle, V\rangle$, it follows that for every $w\in W$:
$$\mathcal{M},w\Vdash\varphi\Leftrightarrow\mathfrak{S}_{\mathcal{M}},S_w\vDash\varphi.$$
\end{thm}
\begin{proof}
Intuitionistic proof-theoretic systems can be treated as Kripke models (\cite{Piecha2016-gt}). This means we can use the Bisimulation Theorem. The result thus follows via Lemma~\ref{lem:findw} and Lemma~\ref{lem:findS} if it can be shown that $\mathcal{M},w\Vdash p\Leftrightarrow\mathfrak{S}_{\mathcal{M}},S_w\vDash p.$
Note that

$$\mathcal{M},w\Vdash p\underset{def.\Vdash}{\Leftrightarrow} p\in V(w)\underset{def.\mathfrak{S}_\mathcal{M}}{\Leftrightarrow} \Bar{p}\in S_w\underline{\Rightarrow} S_w\vdash p\underset{def.\vDash}{\Leftrightarrow}\mathfrak{S}_{\mathcal{M}},S_w\vDash p.$$

\noindent What is left to show is that $S_w\vdash p\Rightarrow \axiom{p}\in S_w$. Assume $S_w\vdash p$ but $\axiom{p}\notin S_w$. It follows that there must be a closed proof of $p$, say $\mathcal{D}$, such that the axioms of $\mathcal{D}$ are not $\axiom{p}$ but the conclusion is $p$. This requires an inference rule $\inference{q_1\cdots q_n}{p}$ where $p$ does not occur among the $q_1,\dots, q_n$. The only candidate for this is $\inference{\bot}{p}$. The use of this rule would require a proof of $\bot$. Because $\mathcal{M}$ does not model $\bot$, the only rule that can contain $\bot$ in the conclusion is therefore $\inference{\bot}{\bot}$. But no proof of $\bot$ can be constructed from this rule. Therefore $\mathcal{D}$ cannot be a proof of $p$ not containing $\axiom{p}$.
\end{proof}

We can now prove our key result:
\begin{thm}
$\varphi$ is a generalised proof-theoretically valid formula $\Leftrightarrow$ $\varphi$ is an intuitionistic validity.
\end{thm}
\begin{proof}
First, assume $\varphi$ is an intuitionistic validity. Then every intuitionistic Kripke model forces $\varphi$ and every intuitionistic proof-theoretic system is equivalent to an intuitionistic Kripke model (\cite{Piecha2016-gt}). For every intuitionistic proof-theoretic system $\mathfrak{S}$, it thus follows that $\mathfrak{S}\vDash\varphi$.

Next, assume that for every intuitionistic proof-theoretic system $\mathfrak{S}$, it follows that $\mathfrak{S}\vDash\varphi$, while for a contradiction, assume that $\varphi$ is not an intuitionistic validity. Then there is a finite intuitionistic Kripke model $\mathcal{M}$ (by Theorem~\ref{thm:finitecomplete}) that does not model $\varphi$. But by Theorem~\ref{thm:kripketopts}, it follows that $\mathfrak{S}_\mathcal{M}\nvDash\varphi$, which contradicts the initial assumption.
\end{proof}

\boldchange{Piecha and Schroeder-Heister came up with very plausible restrictions on any proof-theoretic validity notion. One condition they place is called export, which states that a atomic system can be coded as a set of formulas. What changes here is that validity is now defined relative to all proof-theoretic systems, no longer being relative to all atomic systems.  This would require the following generalisation of export:
\begin{multline}
    \tag{Export}\text{For every proof-theoretic system and atomic system } \mathfrak{S}, S\\ \text{ there is a set of formulas }\Gamma \text{ such that } \mathfrak{S},S\vDash \varphi \Leftrightarrow (\vDash\Gamma\Rightarrow\vDash\varphi) 
\end{multline}
\noindent But, unlike atomic systems, proof-theoretic systems are too complex to be coded by a single set of formulas.

As \textcite[244-5]{Piecha2019} note, this condition is also violated by Goldfarb's (\citeyear{Goldfarb2016-tk}) proof-theoretic validity notion.  Goldfarb provides a system that becomes intuitionistic when closed under substitution. His approach is very similar to the one adopted in this paper, although he uses a different collection of proof-theoretic systems.  In particular, he can be understood as taking every system of the form $\set{S\subseteq \mathbb{S}\mid G\subseteq S}$, where $G$ is some atomic system. His set of proof-theoretic systems is too restrictive for an analogue of Theorem~\ref{thm:kripketopts}, which allows an interpretation of all finite intuitionisitic Kripke models.}

\section{Conclusion}
While \textcite{Piecha2019} demonstrated that Prawitz's conjecture is false for the definition of proof-theoretic validity given by Prawitz in the 1970s, we have demonstrated how a small modification can produce a generalised notion of proof-theoretic validity for which Prawitz's conjecture is true. \boldchange{Moreover, some straightforwards adaptations of our definitions may be used to prove similar results for minimal logic, since our proof generates intuitionistic proof-theoretic systems from intuitionistic Kripke models and could also be used to obtain minimal proof-theoretic systems from minimal Kripke models. (See \cite{De_Jongh2015-ic} and \cite{Colacito2016-pd} for the definition of minimal Kripke models, the finite model property for minimal logic, and other modification that would be needed.)}

Our generalisation may seem motivated by the desired technical result rather than by the underlying philosophy, but we can provide solid motivation for the modification. The following two points are to be given in its favour: 
\begin{itemize}
    \item[] First, it is natural to think of an atomic system $S$ as a possible inferentialist definition for atomic propositions. Once we think of them this way, it is natural to think of proof-theoretic systems as providing us with different ways of defining atomic propositions. But the proof-theoretic system gives more information: it tells us, for instance, whether two ways of defining the atomic propositions are compatible or whether a particular definition can be extended by additional rules. Given that proof-theoretic validity is supposed to capture what is logically valid, we should consider not only every way the atomic propositions might be defined, but also all the different ways definitions might be extendable or incompatible. To do this, we need generalised proof-theoretic validity.
    \item[] Second, an examination of why it is that all proof-theoretic systems are superintuitionistic makes it clear that information is encoded by the atomic rules. Still, proof-theoretic validity is about logical connectives, not about atomic propositions. The natural response to the treatment of the atomic propositions that encode information is to generalise the treatment, as we have done.
\end{itemize} 
\printbibliography
\end{document}